\documentclass[a4paper, 12pt, oneside]{amsart}
\usepackage{amsmath,amscd}
\usepackage{amssymb}
\usepackage{amsthm}
\usepackage{comment}
\usepackage{graphicx, xcolor}
\usepackage[abbrev,nobysame,alphabetic]{amsrefs}
\usepackage{mathrsfs}
\usepackage[ocgcolorlinks, linkcolor=blue]{hyperref}

\usepackage{bm}
\usepackage{bbm}
\usepackage{url}

\usepackage[utf8]{inputenc}
\usepackage{mathtools,amssymb}
\usepackage{esint}
\usepackage{tikz}
\usepackage{dsfont}
\usepackage{relsize}
\usepackage{url}
\usepackage{xcolor}
\usepackage{graphicx}
\usepackage{mathrsfs}
\usepackage[shortlabels]{enumitem}
\usepackage{lineno}
\usepackage{amsmath}
\usepackage{enumitem}
\usepackage{amsthm} 
\usepackage{verbatim}
\usepackage{dsfont}
\numberwithin{equation}{section}

\allowdisplaybreaks

\graphicspath{{images/}}

\newtheorem{thm}{Theorem}[section]

\newtheorem{lem}[thm]{Lemma}
\newtheorem{prop}[thm]{Proposition}

\theoremstyle{definition}
\newtheorem{rem}[thm]{Remark}

\newtheorem*{claim*}{Claim}
\theoremstyle{remark}

\numberwithin{equation}{section}

\newcommand{\R}{{\mathbb R}}

\newcommand{\LC}{\left(}
\newcommand{\RC}{\right)}

\DeclareMathOperator{\supp}{supp} 

\newcommand{\dif}[1]{\,\mathrm{d}{#1}} 

\title[Non-linear fractional inverse scattering]{On a nonlocal non-linear inverse scattering problem}

\author{Saumyajit Das}
\address{Harish-Chandra Research Institute, Homi Bhabha National Institute, Chhatnag Road, Jhunsi, Prayagraj (Allahabad) 211 019, India}
\email{saumyajit.math.das@gmail.com}

\author{Susovan Pramanik}
\address{Harish-Chandra Research Institute, Homi Bhabha National Institute, Chhatnag Road, Jhunsi, Prayagraj (Allahabad) 211 019, India}
\email{susovanpramanik@hri.res.in}


\begin{document}

\maketitle
\begin{abstract}
In this article, we study the inverse scattering problem for the nonlinear fractional Helmholtz equation with cubic nonlinearity in three dimensions, where we recover a compactly supported potential from scattering amplitude.
\end{abstract}

\section{Introduction}

The inverse scattering problem concerns the recovery of unknown potentials or media from scattering amplitude measurements, with applications in quantum optics for modeling nonlocal photon-atom interactions and geophysical electromagnetics for anomalous responses in self-similar formations. Physically, these arise in wave propagation through heterogeneous media exhibiting nonlocal and nonlinear effects, such as fractional elasticity models. For classical background, we refer readers to \cite{ColtonKress1998,kraisler2022collective,kraisler2023kinetic}. Beyond physics, inverse scattering exhibits mathematical richness. For detailed study, we refer to \cite{eskin2011lectures,cakoni2006qualitative,colton1998inverse,isozaki2013recent,uhlmann2000inverse}. 

In recent years, studies have examined the nonlinear Helmholtz equation in the scattering of electromagnetic waves through localized nonlinear Kerr media \cite{baruch2009high,fibich2005numerical,wu2018finite}. Here, the nonlinear Helmholtz equation arises from reducing Maxwell's equations for a linearly polarized electric field after eliminating the magnetic field. In \cite{wu2018finite}, the following equation emerges:
$$
-\Delta u - k^2 u = \rho 1_\Omega |u|^2 u \quad \text{in } \mathbb{R}^n,
$$
where $\Omega$ is the support of the nonlinear Kerr medium and $\rho$ is the Kerr constant. Motivated by the above, in this article, we consider the fractional Helmholtz equation with potential and cubic nonlinearity \eqref{Frac Helmholtz eq}. In particular, when $s=1$, it reduces to the Helmholtz equation, which frequently arises in scattering theory and quantum mechanics \cite{shen2025complex,ColtonKress1998}; the fractional Helmholtz equation also appears in geophysical electromagnetics \cite{weiss2020fractional,glusa2021fast} and quantum optics to describe the nonlocal characteristics of photon-atom interactions \cite{hiltunen2024nonlocal,hoskins2023fast}. Furthermore, when $k=1$ and $s=1$, the equation \eqref{Frac Helmholtz eq} becomes Ginzburg-Landau equation (for details see \cite{gutierrez2004non, furuya2020direct}). In this article, for $k=1$, our model \eqref{Frac Helmholtz eq} exactly becomes fractional Ginzburg-Landau equation. Motivated by its physical relevance across diverse phenomena in physics, we study the inverse scattering problem for the fractional Helmholtz equation to recover the compactly supported potential from the measurement of scattering amplitude. In particular we study the following Helmholtz equation: for $j=1,2$
\begin{align}\label{Frac Helmholtz eq}
    (-\Delta)^s u_j-k^{2s} u_j=Q_j(x) |u_j|^2u_j, \quad \mbox{in}\ \mathbb{R}^d,
\end{align}
associated with the following Sommerfield radiation condition:
\begin{align}\label{Sommerfield radiation}
    \lim\limits_{R\to\infty} \frac 1R \int_{1\leq \|x\| \leq R} | D^s_x u_j^{sc} -k^{2s}u_j^{sc}\hat{x}|^2 \, {\rm{d}}x=0,
\end{align}
where $\hat{x}=\frac{x}{\| x\|}$,  $D_x^s := \mathcal F^{-1} |\xi|^s \mathcal F$ and $u_i^{sc}$ is the scattering part of the solution $u$. For $j=1,2$, the potential $Q_i(x)\in \mathrm{L}^{\infty}(\mathbb{R}^d)$ is assumed to be compactly supported in $\Omega\subset \R^d$ and $s\in(\frac{d}{d+1},1)$.  The solution $u_j$ can be expressed as: for $j=1,2$
\begin{align}\label{representation of fractional solution}
    u_j= u_j^{\rm{in}}+ u_j^{\rm{sc}},
\end{align}
where $u_j^{\rm{in}}$ satisfy the initial wave equation
\begin{align}\label{initial wave}
    (-\Delta)^su_j^{\rm{in}}-k^{2s} u_j^{\rm{in}}=0, \quad \mbox{in}\ \R^d.
\end{align}
given by the Herglotz wave function
\[
u_j^{\rm{in}}=\int_{\mathbb{S}^{d-1}} e^{ikx\cdot\theta} g(\theta) {\rm{d}}\theta,
\]
for any function $g\in\mathrm{L}^2(\mathrm{S}^{d-1})$, whereas $u_j^{\rm{sc}}$ satisfy the scattering part of the solution, given by
\begin{align}\label{scattering wave}
    (-\Delta)^su_j^{\rm{sc}}-k^{2s} u_j^{\rm{sc}}= Q_j |u_j|^2u_j, \qquad \mbox{in}\ \mathbb{R}^d.
\end{align}
We would like to note that the initial wave solutions are the same for both the equations, i.e., $u_1^{\rm{in}}=u_2^{\rm{in}}$. Thanks to the Stein-Tomas theorem \cite{Tomas1975,Stein1993}, we have the following estimate
\begin{align}\label{Stein-Tomas}
    \| u_i^{\rm{in}}\|_{\mathrm{L}^{2\frac{d+1}{d-1}}(\R^d)} \leq C_{ST} \| g\|_{\mathrm{L}^2(\mathbb{S}^{d-1})}, \quad \mbox{for}\ j=1,2,
\end{align}
where $u_i^{\rm{in}}$ is given by \eqref{initial wave} and $C_{ST}$ is a positive constant depending on the dimension.
The following existence result we borrow from \cite[Theorem 1.3]{shen2025complex}:
\begin{thm}[\cite{shen2025complex}]\label{existence result}
    Let $u$ satisfies 
    \begin{align}\label{Frac Helmholtz eq no index}
    (-\Delta)^s u-k^{2s}u= Q(x) |u|^{t-1}u, \qquad \mbox{in} \ \R^d.
    \end{align}
    where $k>0$ and $Q(x)\in \mathrm{L}^{\infty}(\R^d)$ is compactly supported. Let the initial wave function is given by
    \[
    u^{\rm{in}} = \int_{\mathbb{S}^{d-1}} e^{ikx\cdot \theta} g(\theta) \, {\rm{d}}\theta,
    \]
    where $\| g\|_{\mathrm{L}^{2}(\R^d)}\le \epsilon$ for some $\epsilon>0$. Furthermore, assume 
    \[
    s\in \left(\frac{d}{d+1},1\right) \ \text{and}\ \frac{d-1+4s}{d-1} < t < \frac{d+1}{d+1-4s}. 
    \]
    Then, there exists a unique solution `$u$' to the equation \eqref{Frac Helmholtz eq no index}, satisfying 
    \[
    u\in \mathrm{L}^q(\R^{d}), \ \text{such that}\ \| u\|_{\mathrm{L}^q(\R^d)} \leq a,
    \]
    provided $u$ satisfy the following Sommerfield radiation condition:
\begin{align*}
    \lim\limits_{R\to\infty} \frac 1R \int_{1\leq \|x\| \leq R} | D^s_x u^{sc} -k^{2s}u^{sc}\hat{x}|^2 \, {\rm{d}}x=0.
\end{align*}
Here $a>0$ some positive constant depeding on $\epsilon$ and the exponent `$q$' satisfies
    \[
    \frac{d(t-1)}{2s}< q< \frac{2dt}{d+1}.
    \]
In particular, for the spatial dimension $d=3$ and $t=3$, we obtain that 
    \[
    u\in\mathrm{L}^4(\R^d).
    \]
\end{thm}
\rem
The uniqueness follows from the Sommerfeld radiation condition, formulated through the local estimate in \cite[Theorem 1.2]{shen2025complex}. Additionally, when $s=1$, $k=1$, \eqref{Frac Helmholtz eq} reduces to the Ginzburg--Landau equation studied in \cite{gutierrez2004non}. Our nonlinear term in \eqref{Frac Helmholtz eq} derives from the Ginzburg--Landau equation. Building on the work of Guti\'{e}rrez in \cite{gutierrez2004non}, the authors of \cite{shen2025complex} study its fractional version.
 
To state our result, let us define the \emph{scattering amplitude} (or \emph{far-field pattern}) by (see Section~\ref{sec:Scattered Field and Scattering Amplitude} for details)
\begin{equation}
u^\infty(\hat{x}) = \int_{\mathbb{R}^d} e^{-\mathsf{i} k \hat{x} \cdot y} Q(y) |u(y)|^2 u(y) \, dy, \quad \hat{x} \in \mathbb{S}^{d-1}. \label{eq:scattering_amplitude}
\end{equation}
Now the question is: can one recover the potential from a priori knowledge of the scattering amplitude $u^\infty$? Here we solve this problem in dimension $d=3$. In particular, we have following result:

\begin{thm}\label{recovery of potential}
    Let $s\in\left( \frac{4}{5},\frac 32\right)$ and $k>k_0>0$. For $j=1,2$, let $Q_j\in\mathrm{L}^{\infty}(\Omega)$ and compactly supported. Let $u_j$ satisfies: for $j=1,2$
    \[
    (-\Delta)^s u_{j}- k^{2s} u_j=Q_j |u_j|^2u_j \quad \text{ in }  \mathbb{R}^3.
    \]
    The solution $u_j$ is given by
    \[
    u_j= u_j^{\rm{in}}+u_j^{\rm{sc}},
    \]
    where $u_j^{\rm{in}}$ is the incident wave and $u_j^{\rm{sc}}$ is the scattering wave as defined in \eqref{initial wave} and \eqref{scattering wave}. Let the scattering amplitude corresponding to the scattering waves are the same i.e., for all $(k,\hat{x}, \theta)\in (k_0,+\infty)\times\mathbb{S}^{2}\times\mathbb{S}^{2}, \, k_0>0$,
    \begin{equation}
u^{\infty}_1(k,\hat x, \theta)=u^{\infty}_2(k,\hat x, \theta), \label{1.11} 
\end{equation}
where the scattering amplitude as defined in \eqref{eq:scattering_amplitude}. Then \[Q_1=Q_2.\] 
\end{thm}
\rem Due to the lack of resolvent estimates for all fractional exponents $s$ (see Proposition \ref{resolvent estimate}), and the restriction arising from the Sobolev embedding (see Lemma \ref{local integrability estimate}), we restrict our analysis to $s \in \left(\frac{4}{5}, \frac{3}{2}\right)$. For other values of the fractional exponent $s$, the problem remains open.
\rem Note that here we vary all three variables $(k,\hat{x},\theta) \in (k_0,\infty) \times \mathbb{S}^{d-1} \times \mathbb{S}^{d-1}$ in the scattering amplitude to determine the potential $Q$. This corresponds to data of dimension $2d-1$. A more challenging question arises when $k$ is fixed and only $(\hat{x},\theta) \in \mathbb{S}^{d-1} \times \mathbb{S}^{d-1}$ are varied: can the potential $Q$ still be uniquely recovered? In this case, the dimension of the data is $2d-2$. This problem is related to the Ginzburg-Landau equation when $k=1$ is fixed. In the local case, i.e., when $s=1$, this problem has been studied in \cite{furuya2020direct}. To the best of our knowledge, the corresponding scattering problem for the fractional Ginzburg-Landau equation remains open, which corresponds to the case $k=1$ in \eqref{Frac Helmholtz eq}. We state the open problem below.

\textbf{Open problem: Ginzburg-Landau equation}

\begin{itemize}
    \item [$\bullet$]  For $j=1,2$, let $Q_j\in\mathrm{L}^{\infty}(\Omega)$ and compactly supported. Let $u_j$ satisfies: for $j=1,2$
    \[
    (-\Delta)^s u_{j}- u_j=Q_j |u_j|^2u_j \quad \text{ in }  \mathbb{R}^3.
    \]
    The solution $u_j$ is given by
    \[
    u_j= u_j^{\rm{in}}+u_j^{\rm{sc}},
    \]
    where $u_j^{\rm{in}}$ is the incident wave and $u_j^{\rm{sc}}$ is the scattering wave as defined in \eqref{initial wave} and \eqref{scattering wave}. Let the scattering amplitude corresponding to the scattering waves are the same i.e., for all $(\hat{x}, \theta)\in \mathbb{S}^{2}\times\mathbb{S}^{2}$,
    \begin{equation*}
u^{\infty}_1(1,\hat x, \theta)=u^{\infty}_2(1,\hat x, \theta).
\end{equation*}
where the scattering amplitude as defined in \eqref{eq:scattering_amplitude}. Then \[Q_1=Q_2.\] 
\end{itemize}

In recent years, the study of inverse scattering has been quite dynamic. Potential recovery using far-field data in the classical case was studied in \cite{liu2015determining}. For the fractional case, the authors in \cite{das2025inverse} recently studied the inverse scattering problem for the fractional Schr\"{o}dinger operator. For the fractional Helmholtz equation, we refer to \cite{li2026inverse}. For the nonlinear and semilinear cases with local operators, we refer the reader to \cite{lassas2021inverse,furuya2020direct,hogan2023recovery} and references therein. Most of our technique is motivated by the work of \cite{li2026inverse, shen2025complex, das2025inverse}.

\section{Resolvent Estimate and Limiting Absorption Principle}

\begin{prop}\label{local22}
Assume that $\frac{d}{d+1}\leq s<\frac{d}{2}$, and $u$ satisfies
\[
(-\Delta)^s u-k^{2s}u =f,
\]
with $f\in\mathrm{L}^p(\R^d)$ and $p\in(1,\infty)$. Then there exists a constant $C>0$, independent of $k$, such that
\begin{equation} \label{local23}
	\sup\limits_{R \geq 1 / k^s} \Big( \frac 1 R \int_{1<|x|<R} |u^{\rm{sc}}(x)|^2 \dif x \Big)^{1/2}
	\leq Ck^{ds (\frac 1 p - \frac 1 2) - \frac 3 2 s} \|f\|_{\mathrm{L}^{p}(\R^d)},
\end{equation}
whenever $\frac 1 {d+1} \leq \frac 1 p - \frac 1 2 <\frac s 2$ for $d=3,4$ or $\frac 1 {d+1} \leq \frac 1 p - \frac 1 2 < \frac s d$ for $d \geq 5$. Moreover,
\begin{equation} \label{local24}
	\sup\limits_{R \geq 1/k^s} \Big(\frac 1 R \int_{1<|x|<R}
	|D_{x}^{s} u^{\rm{sc}}(x)|^2 \dif x \Big)^{1/2}
	\leq Ck^{ds (\frac 1 p - \frac 1 2) - \frac s 2} \|f\|_{\mathrm{L}^p(\R^d)},
\end{equation}
whenever $\frac 1 {d+1} \leq \frac 1 p - \frac 1 2 < \frac s d$ for $d \geq 3$.
\end{prop}
The following result yields an estimate of the scattering part of the solution.
\begin{prop}[\cite{shen2025complex}]\label{integral estimate, scattering solution}
    Let $s\in\left[\frac{d}{d+1},\frac{d}{2}\right)$, and let $u$ satisfies
    \[
    (-\Delta)^su-k^{2s}u=f,
    \]
    with $f\in\mathrm{L}^p(\R^d)$ and $p\in(1,\infty)$. Let us fix some $1<p<q<\infty$ satisfying
\begin{equation} \label{les_con}
	\frac 2 {d+1} \leq \frac 1 p - \frac 1 q \leq \frac {2s} d, \qquad
	\frac 1 p > \frac {d+1} {2d}, \qquad
	\frac 1 q < \frac {d-1} {2d}.
\end{equation}
Let the scattering field of the solution is given by
\begin{align}\label{scattering field, def}
    u^{\mathrm{sc}}(x)=\int_{\mathbb{R}^d}\Phi_{s}(x,y)f(y)\,\mathrm{d} y, \ \ \text{ for } x \in \mathbb{R}^d,
\end{align}
 Then there exists a positive constant $C_{p,q}>0$, such that
 \[
 \| u^{\rm{sc}}\|_{\mathrm{L}^q(\R^d)} \leq C_{p,q}\big(k^{2s}\big)^{\frac{d}{2s}\left(\frac 1p -\frac{1}{q}-1\right)} \| f\|_{\mathrm{L}^p(\R^d)}.
 \]
 In particular for spatial dimension $d=3$ and for the special choice of exponent $p=\frac 43$ and $q=4$, we have that
 \begin{align}\label{scattering wave, integral estimate, dimension 3}
    \| u^{\rm{sc}}\|_{\mathrm{L}^4(\R^3)} \leq C_{4,\frac 43} \big(k^{2s}\big)^{-\frac{3}{4s}} \| f\|_{\mathrm{L}^{\frac 43}(\R^3)}.  
 \end{align}
\end{prop}
The above proposition is a direct consequence of the following resolvent estimate and limiting absorption principle.
\begin{prop}[\cite{das2025inverse,shen2025complex}]\label{resolvent estimate}
Let $d\geq 3$, $\frac{d}{d+1}\leq s<\frac{d}{2}$,  $1<p<q<\infty$  are Lebesgue exponents satisfying \eqref{les_con} and $\lambda>0$. Then, for all $\epsilon\in(0,1)$, the following estimate holds 
\begin{align} \label{resol_frac_schro}
	\|((-\Delta)^s-(\lambda+\mathrm{i}\epsilon))^{-1}\|_{L^{p}-L^{q}}\leq C|\lambda+i\epsilon|^{\frac{d}{2s}(\frac{1}{p}-\frac{1}{q})-1}.
\end{align}
where the constant $C>0$ is uniform w.r.t. $\lambda$ and $\epsilon$. 
\newline
Furthermore, there exists a $\lambda_0>0$, such that for $\lambda>\lambda_0$, 
\begin{align} \label{resol_frac_schro}
	\sup_{0<\epsilon<1}\|((-\Delta)^s-(\lambda+\mathrm{i}\epsilon))^{-1}\|_{L^{p}-L^{q}}\leq C|\lambda|^{\frac{d}{2s}(\frac{1}{p}-\frac{1}{q})-1},\,\,\lambda>\lambda_0,
\end{align}
where the constant $C(\lambda_0)>0$ is uniform with respect to $\lambda>\lambda_0$. 
\end{prop}
Rest follows from limiting absorption principle, which says, if $u_{\epsilon}$ satisfies
\[
(-\Delta)^s u_{\epsilon} -(\lambda+i\epsilon) u_{\epsilon}=f,
\]
where $f\in\mathrm{L}^p(\R^d)$ with $\frac 1p >\frac{d+1}{2d}$, then
\[
u_{\epsilon}\to u \ \ \ \text{distributionally},
\]
where $u$ satisfies
\[
(-\Delta)^s u -\lambda u=f \qquad \mbox{in}\ \R^d.
\]
Proof of this limiting absorption principle can be found in  \cite{shen2025complex,das2025inverse}. Next we analyze the scattering wave scattering amplitude corresponding to the equation \eqref{Frac Helmholtz eq} through fundamental solution.

\subsection{Fundamental Solution}

Let us investigate the fundamental solutions since they are crucial to our analysis.
Thanks to limiting absorption principle, the fundamental solution $\Phi_{s}(x)$ can be defined by
\begin{equation}\label{phi}
\Phi_{s}(x) := \lim_{\varepsilon \to 0^+}\Phi_{\pm,\, s,\, \varepsilon}(x),
\end{equation}
where $\Phi_{\pm,\,s,\, \varepsilon}$ are given by
\begin{equation}
\Phi_{\pm,\,s,\, \varepsilon}(x) = (2\pi)^{-d} \int\displaylimits_{\mathbb{R}^d} {\frac{e^{\mathsf{i}x\cdot \xi}}{|\xi|^{2s} - (k\pm i\varepsilon)^{2s}}\,\mathrm{d}\xi}.
\end{equation}
$\Phi_{\pm,\,s,\, \varepsilon}$ is a fundamental solution of
\begin{equation}\label{fundamental solution perturbed Helmholtz}
\left[(-\Delta)^s - (k\pm \mathsf{i}\varepsilon)^{2s}\right] \Phi_{\pm,\,s,\, \varepsilon}(x) = \delta_0(x) \quad\text{ for }x\in  \mathbb{R}^d,
\end{equation}
and $\Phi_{s}$ is a fundamental solution of
\begin{equation}\label{fundamental solution fractional Helmholtz}
\left[(-\Delta)^s - k^{2s}\right] \Phi_{s}(x) = \delta_0(x) \quad\text{ for }x\in  \mathbb{R}^d.
\end{equation}
Here $\delta_0$ signifies the Dirac delta measure located at the origin. We have the following lemma from \cite{das2025inverse}.

\begin{lem} \label{lem:eps1}
If $s>0$ and $k>0$, it holds that
\begin{align}
\Phi_{s}(x) 
= \frac{k^{d-2s}}{(2\pi)^d} \bigg[ \mathrm{P.V.}  \int\displaylimits_{\mathbb{R}^d}{\frac{e^{\mathsf{i}kx\cdot \xi}} {|\xi|^{2s}-1} \mathrm{d}\xi} \pm \frac{\mathsf{i}\pi}{2s}\int\displaylimits_{\mathbb{S}^{d-1}}{e^{\mathsf{i}kx\cdot\omega}\,\mathrm{d}S(\omega)} \bigg], \label{2}
\end{align}
where $\mathrm{P.V.}$ stands for the Cauchy principle value.
\end{lem}

Let us denote
\begin{equation}
\mathcal{S}_d(t):= (2\pi)^{-d/2} \int\displaylimits_{\mathbb{S}^{d-1}} {e^{\mathsf{i}te_1\cdot \omega}\,\mathrm{d}S(\omega)},
\end{equation}
where $e_1$ stands for the unit vector with its first component being $1$ and rest $0$, then it can be seen that $\mathcal{S}_d(t)$ is continuous on $[0,\infty)$ and
\begin{equation}
\label{4}
\mathcal{S}_d(t) =
\begin{cases}
	\frac{J_{d/2-1}(t)}{t^{d/2-1}}, &t\neq 0,\\[1mm]
	\frac{2^{-d/2+1}}{\Gamma(d/2)}, & t=0,
\end{cases}
\quad d\ge 2, \quad \text{ and } \quad \mathcal{S}_3(t) = \sqrt{\frac{2}{\pi}}\frac{\sin t}{t}
\end{equation}
where $J_\alpha$ is the usual Bessel's function of the first kind.
\medskip
We continue from \eqref{2} by using the functions $\mathcal{S}_d(t)$ to write further
\begin{align}
\label{3s}
\Phi_{s}(x) =
\frac{k^{d-2s}}{(2\pi)^d} \bigg[ \mathrm{P.V.}  \int\displaylimits_{0}^\infty{\frac{t^{d-1}\mathcal{S}_d(k|x|t)}{t^{2s}-1}\, \mathrm{d}t} \pm \frac{\mathsf{i}\pi}{2s} \;\mathcal{S}_d(k|x|) \bigg], \quad x\in\mathbb{R}^d.
\end{align}
For $s=1$, we get the expression of the fundamental solution for the Helmholtz operator $(-\Delta) - k^{2}$ in $\mathbb{R}^d$ as
\begin{equation}
\label{31}
\Phi_{1}(x) =
\frac{k^{d-2}}{(2\pi)^d} \bigg[ \mathrm{P.V.}  \int\displaylimits_{0}^\infty{\frac{t^{d-1}\mathcal{S}_d(k|x|t)}{t^{2}-1}\, \mathrm{d}t} \pm \frac{\mathsf{i}\pi}{2} \;\mathcal{S}_d(k|x|) \bigg], \quad x\in\mathbb{R}^d.
\end{equation}
For any $d\geq 2$, let us mention the standard asymptotic behavior of $\Phi_{\pm, 1}$, this can be found  in  \cite[Lemma 19.3]{eskin2011lectures}:
\begin{equation}\label{31asy}
\Phi_{1}(x)= k^{\frac{d-3}{2}}\frac{e^{-\frac{\mathsf{i}\pi(d-3)}{4}}}{2^{\frac{d+1}{2}}\,\pi^{\frac{d-1}{2}}}\,\,\frac{e^{\pm \mathsf{i}k|x|}}{|x|^{\frac{d-1}{2}}} +\mathcal{O}\left(\frac{1}{|x|^{\frac{d+1}{2}}}\right), \quad \text{ as } |x| \to \infty.
\end{equation}
In particular for $d=3$, we have
\begin{equation} \label{313}
\Phi_{1}(x)= \frac{1}{4\pi}\,\frac{e^{\pm \mathsf{i}k|x|}}{|x|}, \quad x\in\mathbb{R}^d.
\end{equation}
Next we introduce scattering amplitude corresponding to \eqref{Frac Helmholtz eq}.

\subsection{Scattered Field and Scattering Amplitude}\label{sec:Scattered Field and Scattering Amplitude} 
Let us recall (see \eqref{scattering wave}) that the \emph{scattered field} is given by: for $j=1,2$
\begin{equation}\label{sct_fld}
u^{\mathrm{sc}}_j(x)=\int_{\mathbb{R}^d}\Phi_{s}(x,y)Q_j(y)|u_j(y)|^2u_j(y)\,\mathrm{d} y, \ \ \text{ for } x \in \mathbb{R}^d.
\end{equation}
We use the following result from \cite{martinez2001theory}
\begin{align}
\Phi_{s}(x-y)=\frac{k^{2(1-s)}}{s}\Phi_{1}(x-y) +\frac{\sin{s\pi}}{\pi}\int_{0}^{\infty}\frac{\lambda^{s}(\lambda-\Delta)^{-1}\delta_y(x)}
{\lambda^{2s}-2\lambda^{s}k^{2s}\cos{s\pi}+k^{4s}}\,  \mathrm{d}\lambda,
\end{align}
and we are interested in providing a decay for the second term when $|x|\to \infty.$
Following the ideas presented as in \cite{das2025inverse}, we will show that
\begin{equation}\label{sasy}
\Phi_{s}(x)\sim \frac{k^{2(1-s)}}{s}\Phi_{1}(x),  
\end{equation}
where $f\sim g$ is defined through $\mathrm{L}^2$ integrability at infinity, i.e.,
\begin{equation} \label{decay} f\sim g \Longleftrightarrow \lim_{R\to\infty} \frac{1}{R}\int_{1<|x|< R}|f(x)-g(x)|^2\,\mathrm{d} x=0.
\end{equation}
By using the following estimates
\begin{equation}
\Big|(\lambda-\Delta)^{-1}\delta_y(x)\Big|\leq
\left\{\begin{array}{cl}
	C\lambda^{\frac{d-2}{4}}|x-y|^{-\frac{d-2}{2}}e^{-\sqrt{\lambda}|x-y|}, \, \sqrt{\lambda}|x-y|>1, \\
	C|x-y|^{2-d}, \, \sqrt{\lambda}|x-y|\leq 1,
\end{array}\right.
\end{equation}
a direct computation which can be found in \cite[page 12-13] {huang2018remarks}   yields 
\begin{align*}
\int_{0}^{\infty}\Big|\,\frac{\lambda^{s}(\lambda-\Delta)^{-1}\delta_y(x)}
{\lambda^{2s}-2\lambda^{s}k^{2s}\cos{s\pi}+k^{4s}}\,  \Big|\,\mathrm{d}\lambda\leq C|x-y|^{-d-2s}, \quad \text{if}\,\, |x-y|>1.
\end{align*}
Hence \eqref{sasy} follows. Further by taking \eqref{31asy} into account, we have
\begin{align}\label{sasy2}
\Phi_{s}(x)\sim \frac{k^{2(1-s)}}{s}\,k^{\frac{d-3}{2}}\frac{e^{-\frac{\mathrm{i}\pi(d-3)}{4}}}{2^{\frac{d+1}{2}}\,\pi^{\frac{d-1}{2}}}\,\,\frac{e^{\pm ik|x|}}{|x|^{\frac{d-1}{2}}}.
\end{align}
The following proposition is taken from \cite{das2025inverse}.
\begin{prop}
\label{prop1}
Let $d\geq 2$ and $0 < s<  1$. Then the fundamental solution $\Phi_{s}(x)$ of the fractional Helmholtz equation
$\LC (-\Delta)^s - k^{2s}\RC \Phi_{s}(x) = \delta_0(x)$  in  $\mathbb{R}^d$,  satisfies \eqref{sasy2}, in the sense of \eqref{decay}.
\end{prop} 
As an application, for $j=1,2$, the integral equation corresponding to \eqref{sct_fld} can be can be expressed as
\begin{equation}\label{lp}
u^{\mathrm{sc}}_j(x)\sim\frac{k^{2(1-s)}}{s}\,k^{\frac{d-3}{2}}\frac{e^{-\frac{\mathsf{i}\pi(d-3)}{4}}}{2^{\frac{d+1}{2}}\,\pi^{\frac{d-1}{2}}}\,\int_{\mathbb{R}^d}\frac{e^{\pm \mathrm{i}k|x-y|}}{|x-y|^{\frac{d-1}{2}}}\,Q_j(y)|u_j(y)|^2u_j(y)\,\mathrm{d}y.
\end{equation}
Next we define the \emph{scattering amplitude} or the \emph{far field pattern}, which is given by the following function: for $j=1,2$
\begin{equation}
u^{\infty}_{j,g}(\hat x)=\int_{\mathbb{R}^d}e^{-\mathsf{i}k\hat x\cdot y}Q_j(y)|u_j(y)|^2 u_j(y)\,\mathrm{d}y, \quad \hat x \in \mathbb{S}^{d-1}. \label{1.8}
\end{equation}
Following the calculation as presented in \cite[page 92]{eskin2011lectures}, the scattering field \eqref{lp} and scattering amplitude are related in the following way: for $j=1,2$
\begin{equation}
u^{\mathrm{sc}}_j(x)\sim \frac{k^{2(1-s)}}{s}\,k^{\frac{d-3}{2}}\frac{e^{-\frac{\mathsf{i}\pi(d-3)}{4}}}{2^{\frac{d+1}{2}}\,\pi^{\frac{d-1}{2}}}\,\,\frac{e^{ \pm\mathsf{i}k|x|}}{|x|^{\frac{d-1}{2}}}\, u^{\infty}_{j,g}(\hat x). \label{1.7}
\end{equation}
Hence the solution corresponding to fractional Helmholtz equation \eqref{Frac Helmholtz eq} can be expressed as
\begin{equation}\label{asmp50}
u_j(x) \sim u^{\mathrm{in}}_j(x)+ \frac{k^{2(1-s)}}{s}\,k^{\frac{d-3}{2}}\frac{e^{-\frac{\mathsf{i}\pi(d-3)}{4}}}{2^{\frac{d+1}{2}}\,\pi^{\frac{d-1}{2}}}\,\,\frac{e^{ \pm\mathsf{i}k|x|}}{|x|^{\frac{d-1}{2}}}\, u^{\infty}_{j,g}(\hat x).
\end{equation} 
Note that, for $j=1,2$, the scattering amplitude $u^{\infty}_{j,g}(\hat x)$ corresponds to the Herglotz wave as an incident field, can be viewed as 
\begin{equation}
u^{\infty}_{j,g}(\hat x)=\int_{\mathbb{S}^{d-1}}u^{\infty}_j(k,\hat x,\theta)g(\theta)\, ds(\theta), \quad\text{ for } \hat x:=\frac{x}{|x|}\in \mathbb{S}^{d-1}, \ g \in \mathrm{L}^{2}(\mathbb{S}^{d-1}), \label{sc_amp_H}
\end{equation}    
where  for $(k, \hat x, \theta) \in (0,\infty)\times\mathbb{S}^{d-1}\times \mathbb{S}^{d-1}$,
\begin{equation}\label{sc_amp_3}
u^{\infty}_j(k,\hat x, \theta)=\int_{\mathbb{R}^d}e^{-\mathsf{i}k\hat x\cdot y}Q_j(y)|u_j(y)|^2\underbrace{\left( e^{iky\cdot\theta}+ u^{\mathrm{sc}}_j(y)\right)}_{u_j(y)}\,\mathrm{d}y,
\end{equation}
and the scattering amplitude corresponds to the incident plane waves is $e^{\mathrm{i}kx \cdot \theta}$.

\section{Inverse Problem}

In this section we devout ourselves to recover the potential from the scattering amplitude. We start with the following integral estimate.
\begin{lem}\label{initial regularity, scattering wave}
    Let $s\in\left(\frac{3}{4},\frac{3}{2}\right)$, and let $u$ satisfies
    \[
    (-\Delta)^su-k^{2s}u=f,
    \]
    with $f\in\mathrm{L}^{\frac{4}{3}}(\R^3)$. Let $u^{\rm{sc}}$ be the scattering field of the solution as defined in \eqref{scattering field, def}. Then there exists a positive constant $C_{\frac 43,q}>0$, such that, the scattering field satisfy the following estimate
 \[
 \| u^{\rm{sc}}\|_{\mathrm{L}^{q}(\R^3)} \leq C_{\frac 43,q}\big(k^{2s}\big)^{\frac{3}{2s}\left(\frac{1}p -\frac 1 q-1\right)} \| f\|_{\mathrm{L}^{\frac 43}(\R^3)},
 \]
\end{lem}
This is a direct application of Proposition \ref{integral estimate, scattering solution} with $p=\frac{4}{3}$ and $q$ satisfies
\[
\frac 12< \frac 1 p-\frac{1}{q}<\frac{4s+3}{12}<\frac {2s}{3},
\]
with $s\in(\frac 34, \frac 32)$ in spatial dimension $d=3$. Note that 
\[
4<q<\frac 6{3-2s}.
\]
Our next goal is to obtain a local integrability estimate of the scattering field $u^{\rm{sc}}$. We use the following proposition from \cite{shen2025complex} which yields a local second order regularity estimate.
\begin{prop}[\cite{shen2025complex},Lemma 5.1 ]\label{regularity estimate}
    Let the dimension $d\geq 3$ and let $s\in\left( \frac d{d+1},\frac{d}{2}\right)$. Let $p\in\left(\frac{2(d+1)}{d-1}, \frac{2d}{d-2s}\right)$ and $f\in\mathrm{L}^{p'}(\R^d)\cap\mathrm{L}^q_{\rm{loc}}(\R^d)$, where $p'$ is the H\"older conjugate of $p$, i.e., $\frac{1}{p}+\frac{1}{p'}=1$ and $q\in(1,\infty)$. Let $u$ satisfy
    \[
    (-\Delta)^s u-k^{2s} u= f, \quad \text{in}\ \mathbb{R}^d.
    \]
    Let the scattering field of the solution $u^{\rm{sc}}$ \eqref{scattering field, def} belongs to the space $\mathrm{L}^{q}_{\rm{loc}}(\R^d)$. Then $u^{\rm{sc}}\in W^{2s,q}_{loc}(\R^d)$ and satisfies the following estimate
    \[
    \| u^{\rm{sc}}\|_{W^{2s,q}(B_r(0))} \leq C_{r,d,p',q} \left ( \| u^{\rm{sc}}\|_{\mathrm{L}^{q}(B_{2r}(0))}+\|f\|_{\mathrm{L}^{q}(B_{2r}(0))}\right),
    \]
    here $C_{r,d,p',q}$ is a positive constant depending on the radius $r$, dimension $d$ and the exponents $p',q$ only.
\end{prop}
The above proposition yields the following local integrability  estimate of the solution corresponding to \eqref{Frac Helmholtz eq}.
\begin{lem}\label{local integrability estimate}
    Let $s\in\left( \frac 45, \frac 32\right)$. For $j=1,2$, let $u_j$ satisfies the fractional Helmholtz equation \eqref{Frac Helmholtz eq} in spatial dimension $d=3$. Then for any $r>0$, we have that $u_j\in \mathrm{L}^6(B_r(0))$ for $j=1,2$. 
\end{lem}
\begin{proof}
    For $j=1,2$, $u_j$ satisfies 
    \[
    (-\Delta)^s u_j -k^{2s}u_j= Q_j |u_j|^2 u_j.
    \]
    We have that $Q_j|u_j|^2u_j\in \mathrm{L}^{p'}(\R^3)$ for any $p'\leq \frac 43$. Applying Lemma \eqref{initial regularity, scattering wave}, we have that $u_j^{\rm{sc}}\in \mathrm{L}^q(\R^3)$, where $4< q<\frac{6}{3-2s}$. As $u_j^{\rm{in}}\in \mathrm{L}^{\infty}(\R^3)\cap\mathrm{L}^{4}(\R^3)$, we have that  $u_j^{\rm{in}}\in \mathrm{L}^q(\R^3)$. Applying Proposition \ref{regularity estimate}, we have that 
    \[
    \| u_j^{\rm{sc}}\|_{W^{2s,\frac{2}{3-2s}}(B_r(0))} \leq C_{r,d,p',q} \left ( \| u_j^{\rm{sc}}\|_{\mathrm{L}^{\frac{2}{3-2s}}(B_{2r}(0))}+\left\|Q_j |u_j|^2u_j\right\|_{\mathrm{L}^{\frac{2}{3-2s}}(B_{2r}(0))}\right).
    \]
    Sobolev embedding theorem yields 
    \[
    W^{2s,\frac{2}{3-2s}}(B_r(0)) \hookrightarrow \mathrm{L}^6 (B_r(0)), \quad \forall \, s\in\left(\frac 45, \frac 32\right), \, d=3.
    \]
    Hence, we have $u_j\in \mathrm{L}^6(B_r(0))$ for $j=1,2$ and for all $r> 0$.
\end{proof}

Before begin the proof of the main result we start with the classical Rellich lemma which can be found in e.g. \cite[Lemma 35.2]{eskin2011lectures}.
\begin{lem} \label{Rellich}
Let $k\neq 0$, $f\in \mathrm{L}^2(\mathbb{R}^d)$ and with  compact support.
$u\in \mathrm{L}^2(\mathbb{R}^d)$ solves
\begin{equation}
	\Delta u + k^{2}u =f, \quad\mbox{in }\mathbb{R}^d,
\end{equation}
Then $u$ has a compact support, and $\supp u\subseteq \supp f$.
\end{lem}

\begin{lem} \label{lem:w12}
Let $w_j~(j=1,2)$ be the solution of 
\begin{equation}\label{uv}
	\Delta w_{j}+  k^{2} w_{j} =-Q_j|u_j^{\rm{sc}}+e^{ikx\cdot\theta}|^2(u^{\rm{sc}}_{j}+e^{ikx\cdot\theta})  \quad\text{ in } \mathbb{R}^3,
\end{equation}
where $u_{j}^{\rm{sc}}$'s are the scattered fields solving \eqref{Frac Helmholtz eq} with $s\in\left(\frac 45, \frac 32\right)$, $\theta\in \mathbb{S}^2$, $k>0$ and $Q_j$'s are the bounded potentials with supports in $\overline{\Omega}$, where $\Omega$ is a smooth bounded domain.
Then $w_1 - w_2 \in \mathrm{L}^2(\mathbb R^3)$ with compact support.
\end{lem}

\begin{proof}
Using the fundamental solution of the Helmholtz operator, we write down the solution as
\begin{equation}\label{vjg}
	w_{j}(x)=e^{ikx}-\int_{\mathbb{R}^d}\Phi_{1}(x,y)Q_j(y)Q_j|u_j^{\rm{sc}}+e^{ikx\cdot\theta}|^2(u^{\rm{sc}}_{j}+e^{ikx\cdot\theta})\,\mathrm{d} y, \quad\text{ for } x \in \ \mathbb{R}^d,
\end{equation}
where $\Phi_{1}$ is given in \eqref{31} and \eqref{313}, with its asymptotics in \eqref{31asy} and the scattered field is
\[
w_j^{\rm{sc}}=-\int_{\mathbb{R}^d}\Phi_{1}(x,y)Q_j|u_j^{\rm{sc}}+e^{ikx\cdot\theta}|^2(u^{\rm{sc}}_{j}+e^{ikx\cdot\theta})\,\mathrm{d} y, \quad\text{ for } x \in \ \mathbb{R}^d.
\]
Thanks to Proposition~\ref{integral estimate, scattering solution}, we have that
\[
\| w^{\rm{sc}}_j\|_{\mathrm{L}^4(\R^d)} \leq C_{4,\frac 43} \big(k^{2s}\big)^{-\frac{d}{4}} \left\| Q_j|u_j^{\rm{sc}}+e^{ikx\cdot\theta}|^2(u^{\rm{sc}}_{j}+e^{ikx\cdot\theta})\right\|_{\mathrm{L}^{\frac 43}(\R^d)}<+\infty,
\]
where the finiteness is due to Theorem~\ref{existence result}. Similar to deriving \eqref{1.7}, we can further infer that \cite{eskin2011lectures}
\begin{equation}
	w_{j}(x)\sim -k^{\frac{d-3}{2}}\frac{e^{-\frac{\mathsf{i}\pi(d-3)}{4}}}{2^{\frac{d+1}{2}}\,\pi^{\frac{d-1}{2}}}\,\,\frac{e^{\pm \mathsf{i}k|x|}}{|x|^{\frac{d-1}{2}}}\, u^{\infty}_{j}(k,\hat x,\theta)+\mathcal{O}\left(\frac{1}{|x|^{\frac{d+1}{2}}}\right),\quad j=1,2,
\end{equation}
and using the hypothesis $u^{\infty}_{1}=u^{\infty}_{2}$, we have
$w_{1}-w_{2} \in \mathrm{L}^2(\mathbb{R}^d)$.
\newline
Denote $f := \left(Q_1|u_1^{\rm{sc}}+e^{ikx\cdot\theta}|^2(u^{\rm{sc}}_{1}+e^{ikx\cdot\theta})-Q_2|u_2^{\rm{sc}}+e^{ikx\cdot\theta}|^2(u^{\rm{sc}}_{2}+e^{ikx\cdot\theta})\right)$  has compact support in $\overline{\Omega}$. Furthermore, $f\in\mathrm{L}^2(\R^d)$, thanks to Lemma \ref{local integrability estimate}.
Moreover, $w_{1}-w_{2}$ solves
\begin{equation}\label{h2}
	\Delta (w_{1}-w_{2}) + k^{2} (w_{1}-w_{2}) =f\quad\mbox{in }\mathbb{R}^d,
\end{equation}
so applying Lemma \ref{Rellich} to $(w_{1}-w_{2})$ gives
\begin{equation}\label{w12}
	w_{1}-w_{2} = 0 \quad\mbox{in }\mathbb{R}^d\setminus\overline{\Omega}.
\end{equation}
Let $\Omega\Subset B(0,\sigma)=B$ for some $\sigma>0$. Note that, $(w_{1}-w_{2})\in H^2(B)$ by the elliptic regularity of the solution of \eqref{h2}. Furthermore, thanks to \eqref{w12}, we can conclude $w_1 - w_2 \in H^2_0(B)$.
The proof is done.
\end{proof}

We are ready to prove Theorem \ref{recovery of potential}.

\begin{proof}[Proof of Theorem \ref{recovery of potential}]
Let $m\in\mathbb{R}^3$ be arbitrary.
Choose $l \in\mathbb{R}^3$ such that $m \cdot l = 0$ and $|m + l| > k_0$.
Denote
\begin{equation} \label{rho}
	\rho := m+l, \quad
	k := |\rho| = \sqrt{|m|^2 + |l|^2}, \quad
	\theta := \frac {m-l} k.
\end{equation}
It can be checked that $k \in (k_0,+\infty)$ and $\theta \in \mathbb S^{2}$.
\newline
Let $w_j~(j=1,2)$ solve
\[
(-\Delta - k^2) w_j = Q_j|u_j^{\rm{sc}}+e^{ikx\cdot\theta}|^2(u^{\rm{sc}}_{j}+e^{ikx\cdot\theta}), \ \text{in}\ \mathbb{R}^3.
\]

By Lemma \ref{lem:w12} we see $w_1 - w_2 \in \mathrm{L}^2(\mathbb R^3)$ with compact support.
Then integration by parts give
\begin{align*} 
	&\int_{B} e^{\mathrm{i}\rho\cdot x} \Big(Q_1|u_1^{\rm{sc}}+e^{ikx\cdot\theta}|^2(u_{1}^{\rm{sc}}+e^{ikx\cdot\theta}) -Q_2|u_2^{\rm{sc}}+e^{ikx\cdot\theta}|^2(u_{2}^{\rm{sc}}+e^{ikx\cdot\theta})\Big)\\
	& = \int_{B} e^{\mathrm{i}\rho\cdot x}(\Delta+k^2) (w_1 - w_2) = \int_{B} (\Delta+k^2) e^{\mathrm{i}\rho\cdot x} (w_1 - w_2)
	= 0.\nonumber
\end{align*}
Simplifying the above relation, we  obtain
\begin{align*} 
	\int_{B} e^{\mathrm{i}\rho  \cdot x} \Big(Q_1&
    (1+|u_1^{\rm{sc}}|^2+\bar{u}_1^{\rm{sc}}e^{ikx\cdot\theta}+u_1^{\rm{sc}}e^{-ikx\cdot\theta})(u_{1}^{\rm{sc}}+e^{ikx\cdot\theta})\Big)\\
    &= \int_{B} e^{\mathrm{i}\rho  \cdot x} \Big(Q_2
    (1+|u_2^{\rm{sc}}|^2+\bar{u}_2^{\rm{sc}}e^{ikx\cdot\theta}+u_2^{\rm{sc}}e^{-ikx\cdot\theta})(u_{2}^{\rm{sc}}+e^{ikx\cdot\theta})\Big).
\end{align*}
Simplifying further, we obtain
\begin{align}
    &\int_{B} e^{ix\cdot(\rho+k\theta)} (Q_1-Q_2)\label{larg}\\
    &= \int_B e^{ix\cdot\rho} (Q_2u_2^{\rm{sc}}-Q_1u_1^{\rm{sc}})+ \int_B e^{ix\cdot\rho} (Q_2|u^{\rm{sc}}_2|^2u_2^{\rm{sc}}-Q_1|u^{\rm{sc}}_1|^2u_1^{\rm{sc}})\nonumber\\
    &+ \int_B e^{ix\cdot\rho} (Q_2 |u_2^{\rm{sc}}|^2e^{ikx\cdot\theta}-Q_1 |u_1^{\rm{sc}}|^2e^{ikx\cdot\theta})+\int_B e^{ix\cdot\rho} (Q_2 ({u}_2^{\rm{sc}})^2e^{-ikx\cdot\theta}-Q_1 ({u}_1^{\rm{sc}}
    )^2e^{-ikx\cdot\theta})\nonumber\\
    &+ \int_B e^{ix\cdot\rho}  (Q_2 |u_2^{\rm{sc}}|^2e^{ikx\cdot\theta}-Q_1 |u_1^{\rm{sc}}|^2e^{ikx\cdot\theta})+ \int_B e^{ix\cdot\rho} (Q_2\bar{u}_2^{\rm{sc}}e^{2ikx\cdot\theta}-Q_1\bar{u}_1^{\rm{sc}}e^{2ikx\cdot\theta})\nonumber\\
    & + \int_B e^{ix\cdot\rho} (Q_2u_2^{\rm{sc}}-Q_1u_1^{\rm{sc}}).\nonumber
\end{align}
We show that the each of the terms appearing on the right hand side actually tends to zero as $k\to+\infty$. We present our calculation only for one term and rest will follow in a similar manner. Let us consider the second term:
\begin{align*}
    &\left| \int_B e^{ix\cdot\rho} (Q_2|u^{\rm{sc}}_2|^2u_2^{\rm{sc}}-Q_1|u^{\rm{sc}}_1|^2u_1^{\rm{sc}})\right| 
    \\
    &\leq |B|^{\frac 14}\left(\|Q_1\|_{\mathrm{L}^{\infty}(\R^3)}+\|Q_1\|_{\mathrm{L}^{\infty}(\R^3)}\right) C_{4,\frac 43} \big(k^{2s}\big)^{-\frac{3}{4s}} \Big(\| u_1\|_{\mathrm{L}^{4}(\R^3)}+\| u_2\|_{\mathrm{L}^{4}(\R^3)}\Big),
\end{align*}
where we use the resolvent estimate as described in \eqref{scattering wave, integral estimate, dimension 3} in Proposition \ref{integral estimate, scattering solution}. Furthermore, using the fact that $\| u_j\|_{\mathrm{L}^{4}(\R^3)}<+\infty$ for $j=1,2$ (see Theorem \ref{existence result}), we have that
\[
\left| \int_B e^{ix\cdot\rho} (Q_2|u^{\rm{sc}}_2|^2u_2^{\rm{sc}}-Q_1|u^{\rm{sc}}_1|^2u_1^{\rm{sc}})\right| \to 0 \ \text{as}\ k\to\infty.
\]
Similarly, we can show, all the terms appearing in the right hand side in \eqref{larg} will vanish as $k\to \infty$. Hence from \eqref{larg}, we have that
\[
\lim_{k\to\infty} \int_{B} e^{ix\cdot(\rho+k\theta)} (Q_1-Q_2)= \lim_{k\to\infty} \int_{B} e^{2ix\cdot m} (Q_1-Q_2)=0. 
\]
Since $m$ is arbitrary, we conclude
\[
\int_{\R^3} e^{2ix\cdot m} (Q_1-Q_2)=0, \quad \forall \, m\in\R^3. 
\]
As $m\in\mathbb{R}^3$ was arbitrary, we conclude that $Q_1 = Q_2$.
\end{proof}

\textbf{Acknowledgment:} The authors were funded by the Department of Atomic Energy $($DAE$)$, Government of India. 

The authors sincerely express gratitude to Prof. Tuhin Ghosh (Harish-Chandra Research Institute, Prayagraj, Uttar Pradesh, India) for his valuable insights and feedback, which significantly contributed to improving this article.

\vspace{.3cm}

	\textbf{Data Availability:} Data sharing is not applicable to this article as no datasets were generated or analyzed during the current study.

{\small 
\bibliographystyle{alpha}
\bibliography{ref}}

\end{document}